\documentclass[10pt]{amsart}
     \makeatletter
     \def\section{\@startsection{section}{1}%
     \z@{.7\linespacing\@plus\linespacing}{.5\linespacing}%
     {\bfseries
     \centering
     }}
     \def\@secnumfont{\bfseries}
     \makeatother
\newtheorem{theorem}{Theorem}[section]
\newtheorem{lemma}[theorem]{Lemma}

\theoremstyle{definition}

\newtheorem{example}[theorem]{Example}
\theoremstyle{remark}

\numberwithin{equation}{section}
\usepackage{amsmath,amsthm,amssymb,amsbsy}

\newcommand{\CCC}{\mathcal{C}}

\newcommand{\EEE}{\mathcal{E}}
\newcommand{\FFF}{\mathcal{F}}

\newcommand{\LLL}{\mathcal{L}}

\newcommand{\df}[1]{\,\mathrm{d}#1}                         

\newcommand{\eps}{\varepsilon}                              

\begin{document}

\setlength{\parindent}{0cm}
\setlength{\parskip}{0.5cm}

\title[An extended Novikov-type criterion]{An
  extended Novikov-type criterion for local martingales with jumps}

\author{Alexander Sokol}

\address{Alexander Sokol: Institute of Mathematics, University of
  Copenhagen, 2100 Copenhagen, Denmark}
\email{alexander@math.ku.dk}
\urladdr{http://www.math.ku.dk/$\sim$alexander}



\begin{abstract}
For local martingales with nonnegative jumps, we prove a sufficient
criterion for the corresponding exponential martingale to be a true
martingale. The criterion is in terms of exponential moments of a
convex combination of the optional and predictable quadratic
variation. The result extends earlier known criteria.
\end{abstract}

\maketitle

\noindent

\section{Introduction}
\label{sec:intro}

In \cite{AAN}, Novikov introduced a sufficient criterion for the
exponential martingale of a continuous local martingale to be a
uniformly integrable martingale. In this paper, we prove a similar
result in the case where the local martingale is not continuous, but
is assumed to have nonnegative jumps. The novelty of our criterion
rests in that our result is stronger than previously known results, in
that it combines optional and predictable components and in that our
proof of the criterion demonstrates a straightforward two-step structure. We begin by fixing our notation and
recalling some results from stochastic analysis.

Assume given a filtered probability space $(\Omega,\FFF,(\FFF_t)_{t\ge0},P)$
satisfying the usual conditions, see \cite{PP} for the definition
of this and other probabilistic concepts. For any local martingale
$M$, we say that $M$ has initial value zero if $M_0=0$. For any local
martingale $M$ with initial value zero, we denote by $[M]$ the
quadratic variation of $M$, that is, the unique increasing adapted
process with initial value zero such that $M^2-[M]$ is a local
martingale.

If $A$ is an adapted increasing process with initial value zero, we
say that $A$ is integrable if $EA_\infty$ is finite, and we say that
$A$ is locally integrable if $A^{T_n}$ is integrable for some
localising sequence $(T_n)$, that is, a sequence of stopping times
increasing to infinity. If $A$ is an adapted process with initial
value zero and paths of finite variation, we say that $A$ is locally
integrable if the variation process is locally integrable. Whenever
$A$ is adapted, has initial value zero, is of finite variation and
is locally integrable, there exists a predictable process $\Pi^*_pA$
with those same properties such that $A-\Pi^*_pA$ is a local
martingale, see Definition VI.21.3 of \cite{RW2}. We refer to
$\Pi^*_pA$ as the dual predictable projection of $A$, or simply as the
compensator of $A$.

If $M$ is locally square integrable, it holds that $[M]$ is locally
integrable, and we denote by $\langle M\rangle$ the compensator of
$[M]$. We refer to $\langle M\rangle$ as the predictable quadratic
variation of $M$. It then holds that $M^2-\langle M\rangle$ is a local martingale.

For any local martingale with initial value zero,
there exists by Theorem 7.25 of \cite{HWY} a unique decomposition
$M=M^c+M^d$, where $M^c$ is a continuous local martingale and $M^d$ is
a purely discontinuous local martingale, both with initial value
zero. Here, we say that a local martingale with initial value zero is
purely discontinuous if it has zero quadratic covariation with
any continuous local martingale with initial value zero. We refer to
$M^c$ as the continuous martingale part of $M$, and refer to $M^d$ as
the purely discontinuous martingale part of $M$.

With $M$ a local martingale with
initial value zero and $\Delta M\ge0$, the exponential martingale of $M$, also known as the
Dol{\'e}ans-Dade exponential of $M$, is given by
\begin{align}
  \EEE(M)_t &= \exp\left(M_t-\frac{1}{2}[M^c]_t\right)\prod_{0<s\le t}(1+\Delta M_s)\exp(-\Delta M_s).
\end{align}
The process $\EEE(M)$ is the unique c\`{a}dl\`{a}g
solution in $Z$ to the stochastic differential equation $Z_t = 1 +
\int_0^t Z_{s-}\df{M}_s$, see Theorem II.37 of \cite{PP}. By Theorem 9.2 of
\cite{HWY}, $\EEE(M)$ is always a local martingale with initial
value one. We are interested in sufficient criteria to ensure that
$\EEE(M)$ is a uniformly integrable martingale. This is a classical
question in probability theory, with applications for example in
finance, stochatic differential equations and statistical inference for continuously observed stochastic
processes, see for example \cite{PS}, \cite{TB}, \cite{KS2}, \cite{KA}
or\cite{YK}. For the case when $M$ is continuous,
sufficient criteria ensuring that $\EEE(M)$ is a uniformly integrable
martingale have been obtained in \cite{AAN}, \cite{CS}, \cite{NK2},
\cite{KS3} and \cite{MU}. For the case when $M$ has jumps, see \cite{LM}, \cite{ISS}, \cite{TO}, \cite{JAY} and \cite{KS}.

We now explain the particular result to be obtained in this paper. In
\cite{AAN}, the following result was obtained: If $M$ is a continuous
local martingale with initial value zero and $\exp(\frac{1}{2}[M]_\infty)$ is integrable, then
$\EEE(M)$ is a uniformly integrable martingale. This criterion is
known as Novikov's criterion.  In \cite{NK}, it was shown that for a continuous local martingale
$M$ with initial value zero, the condition
\begin{align}
  \label{eq:KrylovNovikov}
  \liminf_{\eps\to0}\eps\log E\exp\left((1-\eps)\frac{1}{2} [M]_\infty)\right)<\infty
\end{align}
suffices to ensure that $\EEE(M)$ is a uniformly integrable
martingale. This is an extension of the result in \cite{AAN}. And
in \cite{AS}, optimal constants
$\alpha(a)$ and $\beta(a)$ for $a>-1$ were identified such that when $\Delta
M1_{(\Delta M\neq0)}\ge a$, integrability of
$\exp(\alpha(a)[M]_\infty)$ and $\exp(\beta(a)[M]_\infty)$ suffices to
ensure that $\EEE(M)$ is a uniformly integrable martingale, and it was
noted that for the case $a=0$, $\alpha(a)=\beta(a)=\frac{1}{2}$. Thus,
the case where $\Delta M\ge0$ presents a higher level of regularity
than the general case. In this note, we prove that when $\Delta
M\ge0$, the condition
\begin{align}
  \label{eq:FullExtendedNovikov}
  \liminf_{\eps\to0}\eps\log E\exp\left((1-\eps)\frac{1}{2} (\alpha
    [M]_\infty+(1-\alpha)\langle M\rangle_\infty)\right)<\infty
\end{align}
suffices to ensure that $\EEE(M)$ is a uniformly integrable
martingale, thus extending the results of \cite{AAN} and
\cite{NK}. Note that while sufficiency of simple Novikov-type criteria such as
those given in \cite{AS} follow from the results of \cite{LM}, the condition
(\ref{eq:FullExtendedNovikov}) does not. Also, to the best of the
knowledge of the author, the condition (\ref{eq:FullExtendedNovikov})
is the first one obtained applying both the quadratic variation and the predictable quadratic
variation at the same time.

\section{Main results and proofs}
\label{sec:Main}

In this section, we will prove the following theorem.

\begin{theorem}
\label{theorem:main}
Let $M$ be a locally square integrable local martingale with initial value zero and $\Delta
M\ge0$. Fix $0\le \alpha\le1$ and assume that
\begin{align}
  \label{eq:mainSuffCrit}
  \liminf_{\eps\to0}\eps\log E\exp\left((1-\eps)\frac{1}{2}(\alpha [M]_\infty+(1-\alpha)\langle M\rangle_\infty)\right)<\infty.
\end{align}
Then $\EEE(M)$ is a uniformly integrable martingale. If $\alpha=1$,
it is not necessary that $M$ be locally square
integrable. Furthermore, for all $0\le \alpha\le1$, the constant
$1/2$ in (\ref{eq:mainSuffCrit}) is optimal.
\end{theorem}

Optimality of the constant $1/2$ will be shown in Example
\ref{example:Optimal}. We begin by considering the proof of the case $\alpha=1$, where local square
integrability is not required. Our proof in this case rests on the
following two elementary martingale lemmas and the following real
analysis lemma.

\begin{lemma}
\label{lemma:EMUI}
Let $M$ be a local martingale with initial value zero and $\Delta
M\ge0$. Then $E\EEE(M)_\infty\le 1$, and $\EEE(M)$ is a uniformly
integrable martingale if and only if $E\EEE(M)_\infty=1$.
\end{lemma}
\begin{proof}
This follows from the the optional sampling theorem for nonnegative
supermartingales.
\end{proof}

\begin{lemma}
\label{lemma:UILpSuffCrit}
Let $M$ be a local martingale with initial value zero. Let $\CCC$
denote the set of all bounded stopping times. If there exists $a>1$
such that $(M_T)_{T\in \CCC}$ is bounded in $\LLL^a$, then $M$ is a
uniformly integrable martingale.
\end{lemma}
\begin{proof}
As $(M_T)_{T\in \CCC}$ is bounded in $\LLL^a$, $(M_T)_{T\in \CCC}$ is
uniformly integrable. Let $(T_n)$ be a localising sequence such that
$M^{T_n}$ is a uniformly integrable martingale for each $n\ge1$. Let $S$ be a bounded stopping time. Then
$(M_{T_n\land S})_{n\ge1}$ is uniformly integrable as well. As
$M_{T_n\land S}$ converges almost surely to $M_S$, we conclude that
$M_S$ is integrable and that $M_{T_n\land S}$ converges in $\LLL^1$ to
$M_S$. As $M^{T_n}$ is a uniformly integrable martingale,
$EM^{T_n}_S=0$ by the optional stopping theorem, and thus $EM_S=0$. By
Theorem II.77.6 of \cite{RW2}, $M$ is a martingale. And by our
assumptions, $(M_t)_{t\ge0}$ is uniformly integrable, so $M$ is a
uniformly integrable martingale.
\end{proof}

\begin{lemma}
\label{lemma:QVProofIneq}
Let $x\ge0$. It then holds that
\begin{align}
  0&\le \log\frac{1+\lambda x}{(1+x)^\lambda}\le \frac{\lambda(1-\lambda)}{2}x^2\label{eqn:LogIneq1}\quad\textrm{ and }\\
  0&\le \log\frac{(1+x)^a}{1+a x}\le \frac{a(a-1)}{2}x^2\label{eqn:LogIneq2}
\end{align}
for $0\le \lambda\le 1$ and $a\ge1$.
\end{lemma}
\begin{proof}
We first prove (\ref{eqn:LogIneq1}). To prove the lower inequality, it
suffices to argue the $(1+\lambda x)/(1+x)^\lambda\ge1$, which is
equivalent to $1+\lambda x-(1+x)^\lambda\ge0$. Fix $0\le \lambda\le 1$ and define $h_\lambda(x)=1+\lambda
x-(1+x)^\lambda$. Then
$h_\lambda'(x)=\lambda-\lambda(1+x)^{\lambda-1}\ge0$ and
$h_\lambda(0)=0$. This implies $0\le (1+x)^{\lambda-1}\le1$, as
desired, and thus proves the first inequality in (\ref{eqn:LogIneq1}). In
order to prove the second inequality, we define $g_\lambda$ by putting $g_\lambda(x)=\frac{1}{2}\lambda(1-\lambda)x^2-\log(1+\lambda
x)+\lambda\log(1+x)$. We then need to prove $g_\lambda(x)\ge0$. We obtain $g_\lambda(0)=0$ and
\begin{align}
  g'(x)&=\lambda(1-\lambda)x-\frac{\lambda}{1+\lambda x}+\frac{\lambda}{1+x}\notag\\
  &=\frac{\lambda(1-\lambda)x(1+\lambda
    x)(1+x)-\lambda(1+x)+\lambda(1+\lambda x)}{(1+\lambda x)(1+x)}\notag\\
  &=\frac{(\lambda-\lambda^2)(x^2+\lambda x^2+\lambda x^3)}{(1+\lambda x)(1+x)}\ge0,
\end{align}
so $g_\lambda(x)\ge0$ for all $0\le\lambda\le 1$ and $x\ge0$, yielding
the second inequality in (\ref{eqn:LogIneq1}). Next, consider
(\ref{eqn:LogIneq2}). For the lower inequality, note that
$(1+x)^a-(1+ax)\ge0$, so that $(1+x)^a/(1+ax)\ge1$. For the upper
inequality, we may apply (\ref{eqn:LogIneq1}) to obtain
\begin{align}
  \log\frac{(1+x)^a}{1+ax}
  &=a\log\frac{1+x}{(1+ax)^{1/a}}
  \le a\frac{\frac{1}{a}(1-\frac{1}{a})}{2}(ax)^2
  =\frac{a(a-1)}{2}x^2,
\end{align}
for $a\ge1$.
\end{proof}

\textit{Proof of Theorem \ref{theorem:main} for the case $\alpha=1$.}
In this case, we wish to show that when $\liminf_{\eps\to0}\eps\log
E\exp(((1-\eps)/2)[M]_\infty)$ is finite, $\EEE(M)$ is a uniformly
integrable martingale. We first prove that $\EEE(M)$ is a uniformly integrable martingale
under the stronger condition that $\exp((1+\eps)\frac{1}{2}[M]_\infty)$ is
integrable for some $\eps>0$. Fix such an $\eps>0$, and let
$a,r>1$. Applying (\ref{eqn:LogIneq2}) of Lemma
\ref{lemma:QVProofIneq}, we then have
\begin{align}
  \EEE(M)^a_t&=\exp\left(aM_t-\frac{1}{2}a[M^c]_t+\sum_{0<s\le
      t}\log(1+\Delta M_s)^a-a\Delta M_s\right)\notag\\  
  &=\EEE(ar M)_t^{1/r}\exp\left(\frac{a(ar-1)}{2}[M^c]_t+\sum_{0<s\le
      t}\log\frac{(1+\Delta M_s)^a}{(1+ar\Delta M_s)^{1/r}}\right)\notag\\
  &\le \EEE(ar M)_t^{1/r}\exp\left(\frac{a(ar-1)}{2}[M]_t\right).\label{eqn:AlphaOneFirstBound}
\end{align}
Now let $T$ be a bounded stopping time. Note that as $arM$ has nonnegative jumps, $\EEE(arM)$ is a
nonnegative supermartingale and so $E\EEE(arM)_T\le 1$. Let $y=ar$
and let $s$ be the dual exponent to $r$, such that
$s=r/(r-1)$. Applying H{\"o}lder's inequality in (\ref{eqn:AlphaOneFirstBound}), we obtain
\begin{align}
  E\EEE(M)^a_T
  &\le\left(E\exp\left(\frac{y(y-1)}{2(r-1)}[M]_\infty\right)\right)^{1/s}.\label{eqn:SimpleMainLaIneq1}
\end{align}
Next, note that the mapping $y\mapsto y(y-1)$ is increasing for
$y\ge1$. Therefore,
$\inf_{y>r>1}y(y-1)/(2(r-1))=\inf_{r>1}r/2=1/2$, and so there exists $y>r>1$ such that
$y(y-1)/(2(r-1))\le(1+\eps)/2$. Fixing such $y>r>1$ and
putting $a=y/r$, we obtain $a>1$ and (\ref{eqn:SimpleMainLaIneq1}) allows us to conclude
that with the supremum being over all bounded stopping times, we have
\begin{align}
    \sup_TE\EEE(M)^a_T
  &\le\left(E\exp\left((1+\eps)\frac{1}{2}[M]_\infty\right)\right)^{1/s},\label{eqn:SimpleMainLaIneq2}
\end{align}
where the right-hand side is finite by assumption. By Lemma \ref{lemma:UILpSuffCrit}, $\EEE(M)$ is a uniformly
integrable martingale.

Next, we merely assume that  $\liminf_{\eps\to0}\eps\log
E\exp(((1-\eps)/2)[M]_\infty)$ is finite. In particular, for all
$\eps>0$, $\exp(((1-\eps)/2)[M]_\infty)$ is
integrable. Therefore, $[M]_\infty$ is integrable, so $M$ is a
square-integrable martingale and the limit $M_\infty$
exists. Fix $0<\lambda<1$. As $[\lambda M]_t=\lambda^2[M]_t$, we have by our earlier
results that $\EEE(\lambda M)$ is a uniformly integrable
martingale. Using (\ref{eqn:LogIneq1}) of Lemma \ref{lemma:QVProofIneq}, we have
\begin{align}
  1&=E\exp\left(\lambda M_\infty-\frac{\lambda^2}{2}[M^c]_\infty+\sum_{0<t}\log(1+\lambda\Delta
    M_t)-\lambda\Delta M_t\right)\notag\\
  &=E\EEE(M)_\infty^\lambda
    \exp\left(\frac{\lambda(1-\lambda)}{2}[M^c]_\infty+\sum_{0<t}\log\frac{1+\lambda\Delta
      M_t}{(1+\Delta M_t)^\lambda}\right)\notag\\
  &\le E\EEE(M)_\infty^\lambda  \exp\left(\frac{\lambda(1-\lambda)}{2}[M]_\infty\right).\label{eqn:MainLambdaIneq}
\end{align}
Now fix $\gamma\ge 0$. Applying Jensen's inequality
in (\ref{eqn:MainLambdaIneq}) with the concave
function $x\mapsto x^\lambda$ as well as H{\"o}lder's inequality with the dual
exponents $\frac{1}{\lambda}$ and $\frac{1}{1-\lambda}$, we obtain,
with $F_\gamma=([M]_\infty>\gamma)$, that
\begin{align}
  1&\le E\EEE(M)_\infty^\lambda \exp\left(\frac{\lambda\gamma(1-\lambda)}{2}\right)
      +E\EEE(M)_\infty^\lambda  \exp\left(\frac{\lambda(1-\lambda)}{2}[M]_\infty\right)1_{F_\gamma}\notag\\
  &\le(E\EEE(M)_\infty)^\lambda\exp\left(\frac{\lambda\gamma(1-\lambda)}{2}\right)
     +(E\EEE(M)_\infty1_{F_\gamma})^\lambda\left(E\exp\left(\frac{\lambda}{2}[M]_\infty\right)\right)^{1-\lambda}\notag.
\end{align}
By our assumptions, we have that
$\liminf_{\lambda\to1}(E\exp((\lambda/2)[M]_\infty))^{1-\lambda}$ is
finite. Let $c$ denote the value of the limes inferior. By the above,
we then obtain
\begin{align}
  1\le E\EEE(M)_\infty+cE\EEE(M)_\infty1_{([M]_\infty>\gamma)}.
\end{align}
Letting $\gamma$ tend to infinity, we obtain $1\le E\EEE(M)_\infty$,
which by Lemma \ref{lemma:EMUI} shows that $\EEE(M)$ is a uniformly
integrable martingale.
\hfill$\Box$

For the remaining case of $0\le\alpha<1$, we need the following
further inequalities. 

\begin{lemma}
\label{lemma:PQVProofIneq}
Let $x\ge0$. It then holds that
\begin{align}
  0& \le (1+\lambda x)-(1+x)^\lambda \le \frac{\lambda(1-\lambda)}{2}x^2\label{eqn:PredFirstIneq}\quad\textrm{ and }\\
  0& \le (1+x)^a-(1+ax)\le \frac{a(a-1)}{2}x^2,\label{eqn:PredSecondIneq}
\end{align}
for $0\le \lambda\le 1$ and $1\le a\le 2$.
\end{lemma}
\begin{proof}
Fix $0\le \lambda\le 1$. The lower inequality in
(\ref{eqn:PredFirstIneq}) is equivalent to the statement that $(1+\lambda
x)/(1+x)^\lambda\ge1$, which follows from (\ref{eqn:LogIneq1}) of Lemma
\ref{lemma:QVProofIneq}. Next, put
$g_\lambda(x)=\frac{\lambda(1-\lambda)}{2}x^2+(1+x)^\lambda-(1+\lambda
x)$. In order to obtain the upper inequality, we need to prove
$g_\lambda(x)\ge0$. To this end, note that
\begin{align}
  g_\lambda'(x)&=\lambda(1-\lambda)x+\lambda(1+x)^{\lambda-1}-\lambda \quad\textrm{ and }\\
  g_\lambda''(x)&=\lambda(1-\lambda)-\lambda(1-\lambda)(1+x)^{\lambda-2}.
\end{align}
As $g_\lambda''(x)\ge0$, $g_\lambda'(0)=0$ and $g_\lambda(0)=0$, we conclude that $g_\lambda$ is nonnegative and thus
(\ref{eqn:PredFirstIneq}) holds. Next, consider $a$ with $1\le a\le
2$. Using (\ref{eqn:LogIneq2}) of Lemma \ref{lemma:QVProofIneq}, we
find that the lower inequality of (\ref{eqn:PredSecondIneq})
holds. For the upper inequality, define
$h_a(x)=\frac{a(a-1)}{2}x^2+1+ax-(1+x)^a$, we need to prove
$h_a(x)\ge0$. To do so, we note that
\begin{align}
  h_a'(x)&=a(a-1)x+a-a(1+x)^{a-1}\quad\textrm{ and }\\
  h_a''(x)&=a(a-1)-a(a-1)(1+x)^{a-2},
\end{align}
such that $h_a''(x)\ge0$, $h_a'(0)=0$ and $h_a(0)=0$, yielding as in
the previous case that $h_a$ is nonnegative and so we obtain (\ref{eqn:PredSecondIneq}).
\end{proof}

\begin{lemma}
\label{lemma:PQVProofIneq2}
Let $x\ge0$. It then holds that
\begin{align}
  0\le 
  \log\frac{1+\lambda x+(1+\sqrt{1-\alpha}x)^\lambda-(1+\lambda\sqrt{1-\alpha}x)}{(1+x)^\lambda}
  \le \alpha\frac{\lambda(1-\lambda)}{2}x^2\label{eqn:alphalambdalogineq}
\end{align}
for $\alpha,\lambda\in[0,1]$.
\end{lemma}
\begin{proof}
Let $\beta=\sqrt{1-\alpha}$, such that $\alpha=1-\beta^2$. We need to
prove that for $x\ge0$ and $\beta,\lambda\in[0,1]$, it holds that
\begin{align}
  0&\le \log\frac{\lambda(1-\beta)x+(1+\beta
    x)^\lambda}{(1+x)^\lambda}\le (1-\beta^2)\frac{\lambda(1-\lambda)}{2}x^2.\label{alphalambdalog1}
\end{align}
Consider the first inequality in (\ref{alphalambdalog1}). To prove
this, it suffices to show that for $x\ge0$ and $\beta,\lambda\in[0,1]$
it holds that
\begin{align}
    1&\le \frac{\lambda(1-\beta)x+(1+\beta
    x)^\lambda}{(1+x)^\lambda},
\end{align}
which is equivalent to  $\lambda(1-\beta)x+(1+\beta
x)^\lambda-(1+x)^\lambda\ge0$. As this holds for all $x\ge0$, $\lambda\in[0,1]$ and $\beta$ equal to
one, it suffices to prove that the derivative with respect to $\beta$
is nonpositive, meaning that we need to prove $\lambda x \ge
x\lambda(1+\beta x)^{\lambda-1}$. However, this follows as $0\le (1+\beta
x)^{\lambda-1}\le 1$. Thus, the first inequality in
(\ref{alphalambdalog1}) holds. Next, we consider the second
inequality. We need to show that for $x\ge0$ and
$\lambda,\beta\in[0,1]$, it holds that
\begin{align}
   0&\le (1-\beta^2)\frac{\lambda(1-\lambda)}{2}x^2- \log\frac{\lambda(1-\beta)x+(1+\beta
    x)^\lambda}{(1+x)^\lambda}.
\end{align}
By simple substitution, we note that the result holds when $\beta$ is
equal to one, $x\ge0$ and $0\le\lambda\le1$. It therefore suffices
to prove that the derivative with respect to $\beta$ is nonpositive,
meaning that we need to prove that for $x\ge0$ and $\beta,\lambda\in[0,1]$,
\begin{align}
  0&\ge \frac{\lambda x-x\lambda(1+\beta
    x)^{\lambda-1}}{\lambda(1-\beta)x+(1+\beta x)^\lambda}-\beta\lambda(1-\lambda)x^2.
\end{align}
Multiplying by the divisor, which is positive, this is equivalent to
\begin{align}
  0&\le \beta\lambda(1-\lambda)x^2(\lambda(1-\beta)x+(1+\beta x)^\lambda)-(\lambda x-x\lambda(1+\beta x)^{\lambda-1}).
\end{align}
which follows if we can show $1 \le \beta(1-\lambda)x(\lambda(1-\beta)x+(1+\beta
  x)^\lambda)+(1+\beta x)^{\lambda-1}$. As
  $\beta(1-\beta)\lambda(1-\lambda)x^2\ge0$, it thus suffices to show
  that for $x\ge0$ and $\lambda,\beta\in[0,1]$, we have $1 \le
  \beta(1-\lambda)x(1+\beta x)^\lambda+(1+\beta
  x)^{\lambda-1}$. However, as this holds for any $\beta,\lambda\in[0,1]$ when $x$ is zero, we
find that it suffices to show that the derivative with respect to $x$
is nonnegative, so that we need to show 
\begin{align}
  0&\le \beta(1-\lambda)((1+\beta x)^\lambda-x\beta\lambda(1+\beta
  x)^{\lambda-1})+\beta(\lambda-1)(1+\beta x)^{\lambda-2}
\end{align}
for $x\ge0$ and $\beta,\lambda\in[0,1]$. To this end, as
$\beta(1-\lambda)\ge0$, it suffices to show that $0\le (1+\beta x)^\lambda-x\beta\lambda(1+\beta x)^{\lambda-1}-(1+\beta x)^{\lambda-2}$
for $x\ge0$ and $\beta,\lambda\in[0,1]$. To this end, simply note that
\begin{align}
  &(1+\beta x)^\lambda-x\beta\lambda(1+\beta x)^{\lambda-1}-(1+\beta x)^{\lambda-2}\notag\\
  =&(1+\beta x)^{\lambda-2}( (1+\beta x)^2-x\beta \lambda(1+\beta x)-1)\notag\\
  =&(1+\beta x)^{\lambda-2}((1-\lambda)\beta^2x^2+\beta(2-\lambda)x).
\end{align}
As this is nonnegative, the result follows.
\end{proof}

The upper inequality in Lemma \ref{lemma:QVProofIneq} is not
obvious. However, an indication that the constant
$\alpha\frac{\lambda(1-\lambda)}{2}$ is the right one may be obtained
by a simple argument as follows. By the l'H{\^o}pital rule, we have
\begin{align}
  &\lim_{x\to0}\frac{1}{x^2}\log\frac{1+\lambda x+(1+\sqrt{1-\alpha}x)^\lambda-(1+\lambda\sqrt{1-\alpha}x)}{(1+x)^\lambda}\notag\\
  =&\lim_{x\to0}\frac{1}{x^2}\log\frac{\lambda(1-\sqrt{1-\alpha}) x+(1+\sqrt{1-\alpha}x)^\lambda}{(1+x)^\lambda}\notag\\
  =&\lim_{x\to0}\frac{1}{2x}\left(\frac{\lambda(1-\sqrt{1-\alpha})+\sqrt{1-\alpha}\lambda(1+\sqrt{1-\alpha}x)^{\lambda-1}}{\lambda(1-\sqrt{1-\alpha}) x+(1+\sqrt{1-\alpha}x)^\lambda}-\frac{\lambda}{(1+x)}\right).
\end{align}
Identifying a common divisor and applying the l'H{\^o}pital rule
again, we obtain that the above is equal to
\begin{align}
  &\frac{1}{2}\lim_{x\to0}((1-\alpha)\lambda(\lambda-1)(1+\sqrt{1-\alpha}x)^{\lambda-2})(1+x)\notag\\
  +&\frac{1}{2}\lim_{x\to0}(\lambda(1-\sqrt{1-\alpha})+\sqrt{1-\alpha}\lambda(1+\sqrt{1-\alpha}x)^{\lambda-1})\notag\\
  -&\frac{1}{2}\lim_{x\to0}\lambda(\lambda(1-\sqrt{1-\alpha})+\sqrt{1-\alpha}\lambda(1+\sqrt{1-\alpha}x)^{\lambda-1}),
\end{align}
which by elementary calculations is equal to $\alpha\frac{\lambda(1-\lambda)}{2}$, the factor in front of $x^2$ in Lemma \ref{lemma:PQVProofIneq2}.

\textit{Proof of Theorem \ref{theorem:main} for the case $0\le \alpha<1$.}
We consider the case $0<\alpha<1$, the remaining case of
$\alpha=0$ follows by a similar method.

Fix $\eps>0$. We first prove that $\EEE(M)$ is a uniformly integrable martingale
under the stronger condition that $\exp((1+\eps)\frac{1}{2}(\alpha[M]_\infty+(1-\alpha)\langle M\rangle_\infty)$ is
integrable. Let $a,r>1$. Defining $U$ by putting $U_t = ar\sum_{0<s\le t}\log(1+\Delta M_s)-\Delta
M_s$, we have
\begin{align}
  \EEE(M)^a_t&=\exp\left(ar M_t-\frac{1}{2}[arM]_t+U_t\right)^{1/r}
    \exp\left(\frac{a(ar-1)}{2}[M^c]_t\right).\label{eqn:QLCEMaDecomp}
\end{align}
We wish to decompose the first factor in the right-hand side of
(\ref{eqn:QLCEMaDecomp}) in two ways, one involving an
optional increasing factor and one involving a predictable increasing
factor. Put $N^o_t = ar
M_t$. For the optional decomposition, we note that
\begin{align}
   U_t &=\left(\sum_{0<s\le t}\log(1+\Delta N^o_s)-\Delta N^o_s\right)+\sum_{0<s\le t}\log\frac{(1+\Delta M_s)^{ar}}{1+ar\Delta M_s},\label{eqn:UDecomp1}
\end{align}
which yields
\begin{align}
    \exp\left(ar M_t-\frac{1}{2}[ar M]_t+U_t\right)^{\alpha/r}
  &=\EEE(N^o)^{\alpha/r}_t\exp\left(\frac{\alpha}{r}\sum_{0<s\le
      t}\log\frac{(1+\Delta M_s)^{ar}}{1+ar\Delta M_s}\right).\notag
\end{align}
Next, for $0\le \beta<2$, we define $W^\beta_t=\sum_{0<s\le t}(1+\Delta M_s)^\beta-(1+\beta\Delta
M_s)$. Note that the sum is well-defined, increasing and locally integrable by
(\ref{eqn:PredSecondIneq}) of Lemma \ref{lemma:PQVProofIneq}, as $[M]$ is locally integrable by our
assumptions. Therefore, the compensator $V^\beta$ of $W^\beta$ is
well-defined, and is increasing and locally integrable as well. Also note that $(1+\Delta M_s)^\beta=1+\beta\Delta M_s+\Delta
W^\beta$. Further define two local martingales by putting $N^p_t = ar M_t + W^{ar}_t-V^{ar}_t$ and $\bar{N}^p_t =
\int_0^t (1+\Delta V^{ar}_s)^{-1}\df{N}^p_s$, where $\bar{N}^p$ is
well-defined as $\Delta V^{ar}\ge0$ and $(1+\Delta V^{ar}_s)^{-1}$ is predictable and locally
bounded.

We begin by considering some properties of $\bar{N}^p$. First, we observe that
\begin{align}
  \Delta \bar{N}^p_t
  &=\frac{\Delta N^p_t}{1+\Delta V^{ar}_t}
   =\frac{ar\Delta M_t +\Delta W^{ar}_t-\Delta V^{ar}_t}{1+\Delta V^{ar}_t}\notag\\
  &=\frac{(1+\Delta M_t)^{ar}-(1+\Delta V^{ar}_t)}{1+\Delta V^{ar}_t}-1>-1\label{eqn:BarNpJumpIneq}
\end{align}
Furthermore, define $A^{ar}_t = \sum_{0<s\le t}\Delta
V_s^{ar}(1+\Delta V^{ar}_s)^{-1}$. As $\Delta V^{ar}$ is predictable
and nonnegative, the process $A^{ar}$ is well-defined, and is also predictable,
increasing and locally bounded, and $[A^{ar},N^p]_t = \sum_{0<s\le t}\Delta A^{ar}_s\Delta
N^p_s$. By Proposition I.4.49 of \cite{JS}, $[A^{ar},N^p]$ is a local
martingale. As the two local martingales $\int_0^t
A^{ar}_s\df{N^p}_s$ and $[A^{ar},N^p]$ are purely discontinuous and have the same jumps, they are equal by the uniqueness part of Theorem
7.25 of \cite{HWY}, and we thus obtain
\begin{align}
  \bar{N}^p_t &= N^p_t - \int_0^t \frac{\Delta V^{ar}_s}{1+\Delta V^{ar}_s}\df{N^p}_s\notag\\
  &=ar M_t+W^{ar}_t-V^{ar}_t - \sum_{0<s\le t}(1+\Delta V^{ar}_s)^{-1} \Delta V^{ar}_s\Delta N^p_s.
\end{align}
Also, as the function $x\mapsto \log(1+x)-x$ is nonpositive for $x\ge0$ and $V^{ar}$
is increasing, we obtain $\log(1+\Delta V^{ar})-\Delta
V^{ar}\le0$. Combining our observations, we get
\begin{align}
  &ar\log(1+\Delta M_s)-ar\Delta M_s-(\log(1+\Delta\bar{N}^p_s)-\Delta\bar{N}^p_s)\notag\\
  =&ar\log(1+\Delta M_s)- ar\Delta M_s-\left(\log\frac{(1+\Delta M_s)^{ar}}{1+\Delta V^{ar}_s}-\Delta\bar{N}^p_s\right)\notag\\
  =&\Delta W^{ar}_s-\frac{\Delta V^{ar}_s\Delta N^p_s}{1+\Delta V^{ar}_s}+\log(1+\Delta V^{ar}_s)-\Delta V^{ar}_s
  \le \Delta W^{ar}_s-\frac{\Delta V^{ar}_s\Delta N^p_s}{1+\Delta V^{ar}_s},
\end{align}
where the logarithm in first expression is well-defined by (\ref{eqn:BarNpJumpIneq}). This implies
\begin{align}
   U_t&\le \left(\sum_{0<s\le t}\log(1+\Delta\bar{N}^p_s)-\Delta\bar{N}^p_s\right)
          +\sum_{0<s\le t}\Delta W^{ar}_s-\frac{\Delta V^{ar}_s\Delta N^p_s}{1+\Delta V^{ar}_s}\notag\\
   &= \bar{N}^p_t-ar M_t+\left(\sum_{0<s\le t}\log(1+\Delta \bar{N}^p_s)-\Delta \bar{N}^p_s\right)+V^{ar}_t.\label{eqn:UDecomp2}
\end{align}
Also noting that $[\bar{N}^p]_t = [N^p]_t = [ar M]_t$, we obtain the relationship
\begin{align}
\exp\left(ar M_t-\frac{1}{2}[ar M]_t+U_t\right)^{(1-\alpha)/r}
  &\le\EEE(N^p)^{(1-\alpha)/r}_t\exp\left(\frac{1-\alpha}{r}V^{ar}_t\right)\notag.
\end{align}
Combining our results with (\ref{eqn:QLCEMaDecomp}), we obtain
$\EEE(M)_t^a\le \EEE(N^o)^{\alpha/r}_t\EEE(N^p)^{(1-\alpha)/r}X_t$, where the process $X$ is defined by
\begin{align}
  X_t&= \exp\left(\frac{a(ar-1)}{2}[M^c]_t+\frac{\alpha}{r}\sum_{0<s\le t}\log\frac{(1+\Delta M_s)^{ar}}{1+ar\Delta M_s} +\frac{1-\alpha}{r}V_t\right).
\end{align}
Here, note that by (\ref{eqn:LogIneq2}) of Lemma \ref{lemma:QVProofIneq}
and (\ref{eqn:PredSecondIneq}) of Lemma \ref{lemma:PQVProofIneq}, we
have, for $a,r>1$ such that $1\le ar\le 2$, that
\begin{align}
    \sum_{0<s\le t}\log\frac{(1+\Delta M_s)^{ar}}{1+ar\Delta M_s}
  &\le \frac{ar(ar-1)}{2}[M^d]_t\quad\textrm{ and }\\
  V^{ar}_t &\le \frac{ar(ar-1)}{2}\langle M^d\rangle_t,
\end{align}
leading to inequality
\begin{align}
    \EEE(M)_t^a\le \EEE(N^o)^{\alpha/r}_t\EEE(N^p)^{(1-\alpha)/r}
  \exp\left(\frac{a(ar-1)}{2}(\alpha[M]_t+(1-\alpha)\langle M\rangle_t\right).\label{eq:EMaDecomp}
\end{align}
Next, as $\Delta N^o_t\ge0>-1$ and $\Delta\bar{N}^p_t>-1$, $\EEE(N^o)$ and $\EEE(\bar{N}^p)$ are
nonnegative supermartingales, and so for all bounded stopping times $T$, $0\le E\EEE(N^o)_T\le 1$ and $0\le
E\EEE(N^p)_T\le1$. Now let $s$ be the dual exponent of $r$, such that
$s=r/(r-1)$. Noting that
$\frac{1}{r/\alpha}+\frac{1}{r/(1-\alpha)}+\frac{1}{s}=\frac{1}{r}+\frac{1}{s}=1$,we
may then apply H{\"o}lder's inequality for
triples of functions to the inequality (\ref{eq:EMaDecomp}), yielding for any
bounded stopping time $T$ that
\begin{align}
  E\EEE(M)^a_T \le \left(E\exp\left(\frac{y(y-1)}{2(r-1)}(\alpha[M]_\infty+(1-\alpha)\langle M\rangle_\infty\right)\right)^{1/s},
\end{align}
where $y=ar$. This holds for all $a,r>1$ such that $ar\le
2$, and is a bound similar to
(\ref{eqn:SimpleMainLaIneq1}). Proceeding as in the proof of the case
$\alpha=1$, we then obtain as a consequence of Lemma
\ref{lemma:UILpSuffCrit} that $\EEE(M)$ is a uniformly integrable
martingale.

Next, assume that
$\liminf_{\eps\to0}\eps\log E\exp(((1-\eps)/2)\langle
M\rangle_\infty)$ is finite. In particular, for $\eps>0$, $\exp(((1-\eps)/2)\langle M\rangle_\infty)$ is
integrable. In particular,
$\langle M\rangle_\infty$ is integrable. Now, as $M$ is locally
square-integrable by assumption, $[M]$ is locally integrable. Let
$(T_n)$ be a localising sequence such that $[M]^{T_n}$ is integrable. Then, $([M]-\langle M\rangle)^{T_n}$ is a uniformly
integrable martingale, so $E[M]_{T_n}=E\langle
M\rangle_{T_n}$. Letting $n$ tend to infinity, the monotone
convergence theorem shows that $[M]_\infty$ is integrable, so $M$ is a
square-integrable martingale, in particular the limit $M_\infty$ exists.

Now fix $0<\lambda<1$ and define
\begin{align}
  W^\lambda(\alpha)_t
  &=\sum_{0<s\le t} (1+\sqrt{1-\alpha}\Delta
  M_s)^\lambda-(1+\lambda\sqrt{1-\alpha}\Delta M_s).\label{eqn:WLambdaAlphaDef}
\end{align}
Note that by Lemma \ref{lemma:PQVProofIneq}, the terms in the sum in
(\ref{eqn:WLambdaAlphaDef}) are nonpositive and bounded from below by
$-(1-\alpha)\frac{1}{2}\lambda(1-\lambda)(\Delta M_s)^2$. In
particular, we find that $W^\lambda(\alpha)$ is well-defined,
decreasing and integrable. Letting $V^\lambda(\alpha)$ be the compensator of
$W^\lambda(\alpha)$, $V^\lambda(\alpha)$ is then decreasing integrable
as well, and $W^\lambda(\alpha)-V^\lambda(\alpha)$ is a uniformly
integrable martingale. We show that $V^\lambda(\alpha)$ is
continuous. To this end, let $T$ be some predictable
stopping time. By Theorem VI.12.6 of \cite{RW2} and its proof, we have $E\Delta M_T=0$ and $E\Delta (W^\lambda(\alpha)-V^\lambda(\alpha))_T=0$, so that
\begin{align}
  E V^\lambda(\alpha)_T
  &= E W^\lambda(\alpha)_T \notag\\
  &= E( (1+\sqrt{1-\alpha}\Delta M_T)^\lambda-(1+\lambda\sqrt{1-\alpha}\Delta M_T)) \\
  &= E (1+\sqrt{1-\alpha}\Delta M_T)^\lambda-1\ge0,
\end{align}
because of our assumption that $\Delta M\ge0$. Thus, we know now that
as $V^\lambda$ is decreasing, $\Delta V^\lambda_T\le 0$, and from the
above, $EV^\lambda(\alpha)_T\ge0$. We conclude that $\Delta
V^\lambda_T=0$ for all predictable stopping times. Lemma VI.19.2 of
\cite{RW2} then shows that $V^\lambda(\alpha)$ is continuous.

Let $L^\lambda_t = \lambda
M_t+W^{\lambda}(\alpha)-V^{\lambda}(\alpha)$. By our previous
observations, $L^\lambda$ is a uniformly integrable martingale, in particular
the limit $L^\lambda_\infty$ exists. Note that
$(L^\lambda)^c=\lambda M^c$, so it holds that
$[(L^\lambda)^c]_t=\lambda^2 [M^c]_t$. Also note that by continuity of
$V^\lambda(\alpha)$, we have
\begin{align}
  \Delta L^\lambda_t
  &=\lambda\Delta M_t +(1+\sqrt{1-\alpha}\Delta
  M_t)^\lambda-(1+\lambda\sqrt{1-\alpha}\Delta M_t)\notag\\  
  &=(1-\lambda\sqrt{1-\alpha})\Delta M_t+(1+\sqrt{1-\alpha}\Delta M_t)^\lambda-1\notag\\
  &\ge(1-\lambda\sqrt{1-\alpha})\Delta M_t \ge0,
\end{align}
and as $W^\lambda(\alpha)$ has nonpositive jumps, we also have $\Delta L^\lambda_t \le \lambda\Delta
M_t$. Combining these observations, we obtain $[L^\lambda]_t\le \lambda^2
[M]_t$, yielding that $L^\lambda$ is square-integrable. We then also obtain $\langle
L^\lambda\rangle_t\le \lambda^2\langle M\rangle_t$. This implies
\begin{align}
  \alpha[L^\lambda]_\infty+(1-\alpha)\langle L^\lambda\rangle_\infty
  \le \lambda^2(\alpha[M]_\infty+(1-\alpha)\langle M\rangle_\infty),
\end{align}
so by what we already have shown, $\EEE(L^\lambda)$ is a
uniformly integrable martingale. By elementary calculations, we obtain
\begin{align}
  \EEE(L^\lambda)_\infty
  &=\EEE(M)_\infty^\lambda
    \exp\left(\frac{\lambda(1-\lambda)}{2}[M^c]_\infty+\sum_{0<t}\log\frac{1+\Delta
        N_t}{(1+\Delta M_t)^\lambda}-V^\lambda(\alpha)_\infty\right).\notag
\end{align}
By (\ref{eqn:PredFirstIneq}) of Lemma \ref{lemma:PQVProofIneq} and Lemma \ref{lemma:PQVProofIneq2}, we
obtain the two inequalities
\begin{align}
    -V^\lambda(\alpha)_\infty& \le (1-\alpha)\frac{\lambda(1-\lambda)}{2}\langle M^d\rangle_\infty\\
    \sum_{0<t}\log\frac{1+\Delta L^\lambda_t}{(1+\Delta M_t)^\lambda}
    &\le \alpha\frac{\lambda(1-\lambda)}{2}[M^d]_\infty,
\end{align}
so that combining our conclusions, we have
\begin{align}
  1&=E\EEE(L^\lambda)_\infty\notag\\  
  &\le E\EEE(M)^\lambda_\infty
       \exp\left(\frac{\lambda(1-\lambda)}{2}([M^c]_\infty+\alpha[M^d]_\infty+(1-\alpha)\langle
         M^d\rangle_\infty)\right)\notag\\
  &=E\EEE(M)^\lambda_\infty\exp\left(\frac{\lambda(1-\lambda)}{2}(\alpha[M]_\infty+(1-\alpha)\langle M\rangle_\infty)\right),\label{eqn:QLCMainLambdaIneq}
\end{align}
which is a bound similar to (\ref{eqn:MainLambdaIneq}). Therefore,
proceeding as in the proof of the case $\alpha=1$, we obtain as a consequence of Lemma \ref{lemma:EMUI} that
$\EEE(M)$ is a uniformly integrable martingale.
\hfill$\Box$

We take a moment to reflect on the methods applied in the above
proof, and make the following observations. First, while the proof of
the case $0\le \alpha<1$ is more complicated than the proof of the
case $\alpha=1$, both proofs follow very much the same plan: Use
H{\"o}lder's inequality to argue that the result holds in a simple
case where $\frac{1}{2}$ is exchanged with $(1+\eps)\frac{1}{2}$ in the exponent,
then use H{\"o}lder's inequality again to obtain the general
proof. Also, note that the local martingale $\bar{N}^p$ used in the
first part of the proof of the case $0\le \alpha<1$ is related to
general decompositions of exponential martingales, see Lemma II.1 of
\cite{JM}.

The comparatively simple structure of the proof is made possible by
three main factors: The factor $\lambda(1-\lambda)$ present in the real
analysis inequalities allows us to apply H{\"o}lder's inequality in
the second parts of the proofs. Some of these inequalities have been
noted earlier with a factor $1-\lambda$ instead of
$\lambda(1-\lambda)$, compare for example (\ref{eqn:PredFirstIneq})
with (1.2) and (1.3) of \cite{LM}, where the inequalities follow by a
Taylor expansion argument. The more advanced triple-parameter
inequality (\ref{eqn:alphalambdalogineq}) allows us to obtain a
criterion combining the quadratic variation and the predictable
quadratic variation. Finally, the assumption $\Delta M\ge0$, apart
from making most of the real analysis inequalities applicable, also
ensures that the compensator $V^\lambda(\alpha)$ in the second part of
the proof of the case $0\le\alpha<1$ is continuous.

We conclude by giving an example showing that the constant $1/2$
obtained in Theorem \ref{theorem:main} is optimal.

\begin{example}
\label{example:Optimal}
Fix $0\le \alpha\le 1$. Let $\delta>0$. We wish to identify a locally
square-integrable local martingale $M$ with $\Delta M\ge0$ such that
\begin{align}
  \liminf_{\eps\to0}\eps\log E\exp\left((1-\eps)(1-\delta)\frac{1}{2}(\alpha [M]_\infty+(1-\alpha)\langle M\rangle_\infty)\right)<\infty,
\end{align}
while $\EEE(M)$ is not a uniformly integrable martingale. To this end,
fix $a,b>0$. Let $N$ be a
standard Poisson process and let $T_b = \inf\{t\ge0\mid
N_t-(1+b)t=-1\}$. By the path properties of $N$, $T_b$ is almost
surely finite, and $N_{T_b}-(1+b)T_b=-1$. Define $M_t = a(N^{T_b}_t-t\land T_b)$. Similarly to
\cite{AS}, by Lemma \ref{lemma:EMUI}
 and elementary calculations, we find that $\EEE(M)$ is a uniformly
integrable martingale if and only if 
\begin{align}
  E\exp(T_b((1+b)\log(1+a)-a))<1+a.\label{eqn:CounterExampleNoUniform}
\end{align}
Also, for any $c>0$, we have
\begin{align}
  \exp\left(c(\alpha[M]_\infty+(1-\alpha)\langle M\rangle_\infty)\right)
  =&\exp\left(ca^2(\alpha N_{T_b}+(1-\alpha)T_b)\right)\notag\\
  =&\exp\left(ca^2(\alpha((1+b)T_b-1)+(1-\alpha)T_b)\right)\notag\\
  =&\exp\left(T_bca^2(b\alpha+1)\right)\exp\left(-\alpha ca^2\right).
\end{align}
Thus, it suffices to identify $a,b>0$ such that
(\ref{eqn:CounterExampleNoUniform}) holds and such that
\begin{align}
  \liminf_{\eps\to0}\eps\log E\exp\left(T_b\frac{(1-\eps)(1-\delta)a^2}{2}(b\alpha+1)\right)<\infty,
\end{align}
and to do so, it suffices to identify $a,b>0$ such that
\begin{align}
  E\exp(T_b((1+b)\log(1+a)-a))&<1+a\textrm{ and }\label{eqn:CounterExample1}\\
  E\exp\left(T_b\frac{(1-\delta)a^2(b\alpha+1)}{2}\right)&<\infty.\label{eqn:CounterExample2}
\end{align}
Now define $f_b(\lambda)=\exp(-\lambda)+\lambda(1+b)-1$. Noting that
the process $L^\lambda$ defined by $L^\lambda_t=\exp(-\lambda(N_t-(1+b)t)+tf_b(\lambda))$ is a
nonnegative supermartingale, the optional sampling theorem yields
$E\exp(-T_bf_b(\lambda))\le \exp(-\lambda)$. By elementary
calculations, see \cite{AS} for details, we obtain that
(\ref{eqn:CounterExample1}) is satisfied whenever $0<b<a$. Also, note that $f_b$ takes its
minimum at $-\log(1+b)$. Therefore, $-f_b$ takes its maximum at
$\log(1+b)$, and the maximum is $(1+b)\log(1+b)-b$. Thus, it suffices
to choose $0<b<a$ such that $(1-\delta)a^2(b\alpha+1) /2\le
(1+b)\log(1+b)-b$, in particular it suffices to choose $0<b<a$ such that
\begin{align}
    \frac{(1-\delta)a^2}{2} &\le \log(1+b)-b/(1+b).
\end{align}
To do so, first note that by elementary inequalities, we may
pick $a>0$ so small that $a^2/2\le
(1-\delta)^{-1/2}(\log(1+a)-a/(1+a))$. It then suffices to identify
$b\in(0,a)$ such that $\sqrt{1-\delta}(\log(1+a)-a/(1+a))
  \le\log(1+b)-b/(1+b)$, and this is possible as the mapping $x\mapsto \log(1+x)-x/(1+x)$ is
continuous. This yields the desired example.
\end{example}

\bibliographystyle{amsplain}

\begin{thebibliography}{99}

\bibitem{TB} Bj{\"o}rk, T.: {\em Arbitrage theory in continuous
    time}, 3rd edition, Oxford University press, 2009.

\bibitem{CS} Cherny, A. \& Shiryaev, A. N..: On
  criteria for the uniform integrability of Brownian stochastic
  exponentials. In ``Optimal Control and Partial Differential
  Equations'', 2001.

\bibitem{HWY} He, S.-W., Wang, J.-G. \& Yan, J.-A.:
  Semimartingale Theory and Stochastic Calculus, Science Press, CRC
  Press Inc., 1992.

\bibitem{ISS} Izumisawa, M., Sekiguchi, T., and Shiota, Y.: Remark on a characterization of
BMO-martingales. {\em T{\^o}hoku Math. J.} {\bf 31} (3) (1979) 281--284.

\bibitem{JS} Jacod, J \& Shiryaev, A.: {\em Limit Theorems for
    Stochastic Processes}, Springer-Verlag, 2003.

\bibitem{KS} Kallsen, J. and Shiryaev, A. N.: The cumulant process and Esscher's change of
measure. {\em Finance Stoch.} {\bf 6} (4) (2002), 397--428.

\bibitem{KS2} Karatzas, I. and S. E. Shreve: {\em Brownian Motion and
    Stochastic Calculus}, Springer-Verlag, 1988.

\bibitem{KA} Karr, A.: {\em Point Processes and their Statistical
    Inference}, Marcel Dekker, Inc., 1986.

\bibitem{NK2} Kazamaki, N.: {\em Countinuous Exponential Martingales and BMO}, Springer, 1994.

\bibitem{KS3} Kazamaki, N. \& Sekiguchi, T.: Uniform integrability of
  continuous exponential martingales. {\em Tohoku Math. J. (2)\/}~{\em
    35\/}(2) (1983) 289--301.

\bibitem{NK} Krylov, N.: A simple proof of a result of
  A. Novikov. Preprint, 2009. URL: http://arxiv.org/abs/math/0207013.

\bibitem{YK} Kutoyants, Y. A..: {\em Statistical inference for Ergodic
  Diffusion Processes}, Springer-Verlag, 2004.

\bibitem{LM} L{\'e}pingle, D. and M{\'e}min, J.: Sur
  l'int{\'e}grabilit{\'e} uniforme des martingales
  exponentielles. {\em Z. Wahrsch. Verw. Gebiete} {\bf 42}(3) (1978)
  175--203.

\bibitem{JM} M{\'e}min, J.: D{\'e}compositions multiplicatives de
  semimartingales exponentielles et applications. {\em S{\'e}minaire
    de Probabilit{\'e}s (Strasbourg), tome XII} (1978) 35--46.

\bibitem{MU} Mijatovic, A. \& Urusov, M.: On the martingale property
  of certain local martingales. To appear, 2010.

\bibitem{AAN} Novikov, A. A.: On an identity for stochastic
integrals, {\em Theor. Probability Appl.} {\bf 17} (1972) 717--720.

\bibitem{TO} Okada, T.: A criterion for uniform integrability of
  exponential martingales. {\em T{\^o}hoku Math. J.} {\bf 34} (4) (1982) 495--498.

\bibitem{PS} Protter, P. E. and Shimbo, K.: No Arbitrage and General Semimartingales. In ``A festschrift for Thomas G. Kurtz'', 2008.

\bibitem{PP} Protter, P.: {\em Stochastic Integration and Differential
  Equations}, 2nd edition, Springer, 2005.

\bibitem{RW2} Rogers, L. C. G. and Williams, D.: {\em Diffusions, Markov
  Processes and Martingales}, volume 2, Cambridge University Press, 2000.

\bibitem{AS} Sokol, A.: Optimal Novikov-type criteria for local
  martingales with jumps, Preprint, 2012. URL: http://arxiv.org/abs/1206.7009.

\bibitem{JAY} Yan, J.-A.: {\`A} propos de l'int{\'e}grabilit{\'e} uniforme des
  martingales exponentielles. {\em S{\'e}minaire de probabilit{\'e}s (Strasbourg)} {\bf 16} (1982) 338--347.

\end{thebibliography}

\end{document}